\documentclass[a4paper]{amsart}

\usepackage{mathrsfs}
\usepackage{amsmath,amssymb,latexsym,amsfonts,amscd}
\usepackage{hyperref}

\usepackage{cite}
\usepackage{amsthm}   
\usepackage{empheq}   
\usepackage{graphicx}
\usepackage{color}
\usepackage{titletoc}
\usepackage[titletoc]{appendix}
\usepackage{enumitem}
\usepackage{tikz}
\usetikzlibrary{shapes.geometric, arrows}
\usetikzlibrary{decorations,decorations.markings}
\usetikzlibrary{arrows,shapes,chains}

\numberwithin{equation}{section}

\def\ss{\mathtt{s}}

\newtheorem{thm}{Theorem}[section]

\newtheorem{cor}[thm]{Corollary}

\theoremstyle{remark}

\begin{document}

\title{Markov chain approximations for one dimensional diffusions}

\author{Xiaodan Li}
\address{Fudan University, Shanghai 200433, China.}
\email{xdli15@fudan.edu.cn}

\author{Jiangang Ying}
\address{Fudan University, Shanghai 200433, China.}
\email{jgying@fudan.edu.cn}

\thanks{ The second named author is partially supported by NSFC No. 11871162.}



\begin{abstract}
The Markov chain approximation of a one-dimensional symmetric diffusion is investigated in this paper. Given an irreducible reflecting diffusion on a closed interval with scale function $s$ and speed measure $m$, the approximating Markov chains are constructed explicitly through the trace of the Dirichlet form corresponding to the diffusion. One feature of our approach is that it does not require uniform ellipticity on diffusion coefficient of the limit object or uniform regularity on conductances of the approximative Markov chains, as imposed usually in the previous related works.

\end{abstract}

\maketitle

\tableofcontents
\section{Introduction}\label{Section 1}
In this paper, let $I\subset\mathbb{R}$ be a closed interval with Radon measure $m$, for simplicity $I=[0,1]$. We are concerned with the discrete approximation of an irreducible $m$-symmetric diffusion, or equivalently an irreducible, strongly local, regular Dirichlet form on $L^2(I)=L^2(I,m)$.
An irreducible strongly local regular Dirichlet form on $L^2(I)$ can be represented by
the following Dirichlet form $(\mathscr{E}^{(\ss)},\mathscr{F}^{(\ss)})$  on $L^2(I)$,
\begin{align}\label{eqd}
\begin{split}
&\mathscr{F}^{(\ss)} =\bigg\{u\in L^2(I): u\ll \ss, \dfrac{du}{d\ss}\in L^2(I,d\ss)\bigg\},\\
&\mathscr{E}^{(\ss)}(u,v) =\frac{1}{2}\int_I\frac{du}{d\ss}\frac{dv}{d\ss}d\ss, \qquad u,v\in \mathscr{F}^{(\ss)},
\end{split}
\end{align}
where $\ss$ is a continuous and strictly increasing function on $I$, called a scale function on $I$, $u\ll \ss$ means that the function $u$ is absolutely continuous with respect to $\ss$, and $L^2(I,d\ss)$ is the space of square integrable functions with respect to $d\ss$ on $I$.

It is known from \cite{XPJ} that \eqref{eqd} is 
associated with a diffusion $X$ on $I$, reflected at the boundary, with scale function $\ss$ and speed measure $m$. Actually a one dimensional locally conservative (or no killing inside) irreducible diffusion is characterized by its scale function and speed measure. Classic literatures on one dimensional diffusions are referred to \cite{KH}\cite{LD}. See \cite{MYM} for details about theory of Dirichlet forms.

Stochastic processes, especially diffusions, play an active role in financial and physical models. How to simulate the stochastic processes by Markov chains is an important question in practical applications. Therefore, for one-dimensional diffusions, in this paper we give a natural way to construct the approximative Markov chains, that is , through traces of $(\mathscr{E}^{(\ss)},\mathscr{F}^{(\ss)})$ on approximative lattices. Intuitively, traces of $(\mathscr{E}^{(\ss)},\mathscr{F}^{(\ss)})$ record the trajectory information of $X$ on the lattices. So it is direct and efficient to construct the Markov chains through the trace method.

For diffusion processes on $\mathbb{R}^d$, the classic results are Donsker's invariance principle \cite{M} for Brownian motion and \cite{DS} for diffusion processes in non-divergence form. In \cite{DW}, Stroock and Zheng solve the problem for diffusions corresponding to uniformly elliptic operators in divergence form via Markov chains with finite range and certain uniform regularity. In \cite{RT}, Bass and Kumagai extend the results in two ways: Markov chains with unbounded range are allowed and the strong uniform regularity condition on conductances is weakened to a uniform finite second moment condition. For both \cite{DW}\cite{RT}, a crucial step is to obtain a priori heat kernel estimate of the Markov chains to deduce tightness. In \cite{KZ}\cite{KZ2}, Burdzy and Chen give the discrete approximation for reflecting Brownian motion in a general bounded domain. They use a Dirichlet form based approach to obtain tightness. However, the novelty of our result is that we do not impose conditions like above, precisely it differs from the works above in the following aspects.
\begin{itemize}
\item[1.] In former works, the diffusion coefficient of the limit object is assumed to be continuous and uniformly elliptic. In our setting, the infinitesimal generator is
    $${1\over 2}{d\over dx}{d\over d\ss},$$
    and the diffusion coefficient is $1/\ss'$, when $\ss$ is smooth. The conditions above are satisfied only when $\ss^\prime$ is continuous, bounded away from 0 and above.
We stress here that we allow the discontinuity, degeneracy, singularity and even non-existence of the diffusion coefficient.
\item[2.] In former works, the conductances of the approximative Markov chains are required to satisfy certain uniform regularity,
which does not necessarily hold for the conductance constructed in our approach.
    
\item[3.] In former works, the conductances of the approximating Markov chains are required $L^1_{loc}$ converging to the diffusion coefficient. Our conductances are pointwise convergent and may not satisfy the $L^1_{loc}$ convergence.
\end{itemize}

 We provide here a typical example to show the generality and power of the main result (Theorem \ref{thmwc}). Fix a strictly increasing and absolutely continuous function $\ss$ on $I$ satisfying 
 \[
 \ss'(x)=0~\text{or}~1~\text{a.e.}
 \]
Let $G :=\{x\in I:\ss'(x)=1\}$. Then $G$ is defined in the sense of almost everywhere and it holds that 
$m\big(G\cap (a,b)\big)>0, \quad \forall (a,b)\subset\mathbb{R}.$ Note that this is equivalent to that $\ss$ is strictly increasing. Furthermore, assume that $m(G^c)>0$. In fact, the typical example of $G^c$ is a generalized Cantor set. Therefore, $(\mathscr{E}^{(\ss)},\mathscr{F}^{(\ss)})$ in \eqref{eqd} is a proper regular Dirichlet subspace (see \cite{LJ}) of Brownian motion, which corresponds to the $m$-symmetric diffusion $X$ with singular diffusion coefficient. For $n\geq 1$, the Dirichlet form $(\mathscr{E}^n,\mathscr{F}^n)$ of the approximating Markov chain $X^n$ is constructed as follows:
\begin{align}\label{mc}
\begin{split}
&\mathscr{E}^n(\varphi,\varphi)=\frac{1}{4}\sum_{x\sim y}\big(\varphi(x)-\varphi(y)\big)^2\dfrac{1}{|\ss(x)-\ss(y)|} ,\\
&\mathscr{F}^n=\{\varphi\in L^2(I_n,m_n) : \mathscr{E}^n(\varphi,\varphi)<\infty\},
\end{split}
\end{align}
where $x\sim y$ means $x$ and $y$ are neighbored. $I_n$ and $m_n$ are defined in \S\ref{s2}. Clearly, the conductance in \eqref{mc} does not satisfy the uniform regularity  condition (A1) in \cite{RT}.

The idea to prove the approximation, borrowed from \cite{P}, is to prove Mosco convergence
and tightness. However in the proof of Mosco convergence, the dimension plays
the key role while the proof of tightness is relatively general.

The rest of the paper is organized as follows. In \S\ref{s2}, we provide the construction of a sequence of approximative Markov chains for the given diffusion corresponding to \eqref{eqd}, and then give the precise statement of the main weak convergence result (Theorem \ref{thmwc}). In what follows, the proof of Theorem \ref{thmwc} will be divided into two parts. In \S\ref{s3}, we prove the Mosco convergence of the associated Dirichlet forms of $X^n$. In \S\ref{s4}, we complete the proof of weak convergence by providing some tightness results.

{\it Notations.}
The notation `:=' is read as `to be defined as'. Given a domain $D\subset\mathbb{R},$ the families $C(D),C_c(D)$, and $C_c^\infty(D)$ are those of all continuous functions on $D$, all continuous functions on $D$ with compact support and all smooth functions on $D$ with compact support respectively. The notation $\|\cdot\|_\infty$ means the supremum norm of a bounded function. $\mathbb{E}_x^n$ (resp.$\mathbb{E}_\xi^n) $ and $\mathbb{P}_x^n$ (resp.$\mathbb{P}_\xi^n)$ means expectation  and probability with respect to $X^n$ with starting point $X_0=x$ (resp. initial distribution $\xi$). For $t\geq 0,\lambda>0$, $T_t^n$ and $G_\lambda^n$ are the semigroup and resolvent associated with $X^n$. Similar notations ($\mathbb{E}_x,\mathbb{P}_x,\mathbb{E}_\xi,T_t,G_\lambda)$ are understood in the same way for $X$. Given $T\subset\mathbb{R}^+$, let
$$\mathbb{D}_DT :=\{f:T\mapsto D | f \mbox{ is right continuous having left limits}\}.$$

\section{Main results}\label{s2}
Now we will construct a sequence of Markov chains to approximate the diffusion process $X$ corresponding to \eqref{eqd}.
Take a partition of $I$: $0=a_1<a_2<\cdots <a_{n+1}=1$, set $I_n=\{a_i\}_{i=1}^n$ and $\Delta_i=a_{i+1}-a_{i}$, $1\le i\le n$. Let the sequence of partitions satisfy $\Delta^n :=\max\limits_i\Delta_i\rightarrow 0$ as $n\rightarrow\infty$.
Define point measure on $I_n$ as $m_n(\{a_i\})=m([a_i,a_{i+1})), 1\leq i<n$. It is known that $m_n$ is a Radon smooth measure supported on $I_n$ .

Denote the time-changed process of $X$ with respect to $m_n$ by $X^n$. Then $X^n$ is corresponding to a $m_n$- symmetric regular Dirichlet form $(\mathscr{E}^n,\mathscr{F}^n)$ on $L^2(I_n, m_n)$, that is, the traces of $(\mathscr{E}^{(\ss)},\mathscr{F}^{(\ss)})$ on $I_n$. Precisely, let $\sigma_n$ be the hitting time of $I_n$ relative to $X$. The extended Dirichlet space of $\mathscr{F}^{(\ss)}$ is 
$$\mathscr{F}_e^{(s)} :=\bigg\{u: u\ll \ss, \dfrac{du}{d\ss}\in L^2(I,d\ss)\bigg\}.$$
We consider the following orthogonal decomposition of the space (not necessarily Hilbert space ) $(\mathscr{F}^{(\ss)}_e,\mathscr{E}^{(\ss)})$:
\begin{equation}\label{eqo}
\mathscr{F}^{(\ss)}_e=\mathscr{F}^{(\ss)}_{e,I\backslash I_n}\oplus\mathscr{H}_{I_n}^{(s)},
\end{equation}
where $\mathscr{F}^{(\ss)}_{e,I\backslash I_n}=\{u\in \mathscr{F}^{(\ss)}_e: u=0 \mbox{ q.e. on } I_n\}$. We know that for any $u\in \mathscr{F}^{(\ss)}_e$, $H_{I_n}u(x)=\mathbb{E}_x\big(u(X_{\sigma_n})\big)$ gives the probabilistic expression of the orthogonal projection of $u$ on the space $\mathscr{H}_{I_n}^{(s)}$ and accordingly 
$$
\mathscr{H}_{I_n}^{(s)}=\{H_{I_n}u : u\in\mathscr{F}_e^{(\ss)}\}.
$$
From \cite[$\S$6.2]{MYM}, we have
\begin{align*}
&\mathscr{F}^n=\{\varphi\in L^2(I_n,m_n):\varphi=u \mbox{ $m_n$-a.e. on } I_n \mbox{ for some } u\in\mathscr{F}_e^{(s)}\},\\
&\mathscr{E}^n(\varphi,\varphi) =\mathscr{E}^{(s)}(H_{I_n}u, H_{I_n}u), \quad \varphi\in\mathscr{F}^n, \varphi=u \mbox{ $m_n$-a.e. on } I_n \mbox{ for some } u\in\mathscr{F}_e^{(s)}.
\end{align*}
In the following we claim that 
\begin{align}\label{eqc1}
\begin{split}
&\mathscr{E}^n(\varphi,\varphi)=\frac{1}{2}\sum_{x,y\in I_n}\big(\varphi(x)-\varphi(y)\big)^2C_{x,y}^n,\\
&\mathscr{F}^n=\{\varphi\in L^2(I_n,m_n) : \mathscr{E}^n(\varphi,\varphi)<\infty\},
\end{split}
\end{align}
where $C^n : I_n\times I_n\mapsto\mathbb{R}^+$, is the conductivity function, satisfying $C^n_{x,y}=C^n_{y,x}$ and $C^n_{x,x}=0$ for $x,y\in I_n$. And
\begin{equation}\label{eqc2}
C_{x,y}^n=
\begin{cases}
\dfrac{1}{2|\ss(x)-\ss(y)|} &\mbox{ if $x$ and $y$ are neighbored}; \\
0&\mbox{ otherwise.}
\end{cases}
\end{equation}

In fact, fix $\varphi\in L^2(I_n,m_n), \varphi=u \mbox{ $m_n$-a.e. on } I_n \mbox{ for some } u\in\mathscr{F}_e^{(s)}$. If $x\in [a_i, a_{i+1}]$ for some $1\leq i\leq n-1$. From the continuity of the trajectories of $X$, it follows that 
$$X_{\sigma_n}=a_i\mbox{ or } a_{i+1}, \quad\mathbb{P}_x\mbox{-a.s.}$$
Hence, 
\begin{align*}
H_{I_n}u(x)&=\varphi(a_i)\cdot\mathbb{P}_x(X_{\sigma_n}=a_i)+\varphi(a_{i+1})\cdot\mathbb{P}_x(X_{\sigma_n}=a_{i+1})\\
&=\varphi(a_i)\dfrac{\ss(a_{i+1})-\ss(x)}{\ss(a_{i+1})-\ss(a_i)}+\varphi(a_{i+1})\dfrac{\ss(x)-\ss(a_i)}{\ss(a_{i+1})-\ss(a_i)}\\
&=\varphi(a_i)+\dfrac{\varphi(a_{i+1})-\varphi(a_i)}{\ss(a_{i+1})-\ss(a_i)}\big(\ss(x)-\ss(a_i)\big).
\end{align*}
If $x\in [a_n, a_{n+1}]$, $H_{I_n}u(x)=\varphi(a_n)$. So we get $H_{I_n}u\ll \ss$. And
\begin{equation*}
\dfrac{dH_{I_n}u}{d\ss}(x)=\dfrac{\varphi(a_{i+1})-\varphi(a_i)}{\ss(a_{i+1})-\ss(a_i)}, \quad x\in (a_i,a_{i+1}), 1\leq i\leq n-1.
\end{equation*}
In this way
\begin{align*}
\mathscr{E}^n(\varphi,\varphi) &=\mathscr{E}^{(s)}(H_{I_n}u, H_{I_n}u)=\frac{1}{2}\int_I\bigg(\dfrac{dH_{I_n}u}{d\ss}\bigg)^2d\ss\\
&=\frac{1}{2}\sum_{i=1}^n\int_{a_i}^{a_{i+1}}\bigg(\frac{dH_{I_n}u}{d\ss}\bigg)^2d\ss=\frac{1}{2}\sum_{i=1}^{n-1}\dfrac{\big(\varphi(a_{i+1})-\varphi(a_i)\big)^2}{\ss(a_{i+1})-\ss(a_i)}.\\
\end{align*}
So \eqref{eqc1} is hold. \eqref{eqc1} is associated with the Markov chain $X^n$ that stays at a state $x$ for an exponential length of time with parameter $\lambda^n(x) :=\sum\limits_{z\neq x}C^n_{x,z}/m^n(x)$ and then jumps to the neighbor $y$ with probability $C^n_{x,y}/\big(\sum\limits_{z\in I_n}C^n_{x,z}\big)$.

\begin{thm}\label{thmwc}
The continuous-time Markov chains $\{(X^n,\mathbb{P}_{m_n}^n);n\geq 1\}$ on $I_n$ associated with Dirichlet form \eqref{eqc1} in which conductances are set by \eqref{eqc2} converges weakly to $(X,\mathbb{P}_m)$ associated with \eqref{eqd} on $\mathbb{D}_{I}[0,\infty)$ equipped with the Skorohod topology.
\end{thm}
The proof is provided in the following sections.

\section{Mosco convergence}\label{s3}
We shall prove the Mosco convergence. The generalized version of Mosco convergence from the appendix of \cite{ZPT} is included here for handy reference.  Refer to \cite{KT} for more details about Mosco convergence.

For  $n\geq 1$, $(\mathcal{H}_n, \left\langle \cdot,\cdot\right\rangle_n)$ and $(\mathcal{H}, \left\langle\cdot,\cdot\right\rangle)$ are Hilbert spaces with the corresponding norms $\|\cdot\|_n$ and $\|\cdot\|$. Suppose that $(\mathcal{E}^n,\mathcal{D}(\mathcal{E}^n))$ and $(\mathcal{E},\mathcal{D}(\mathcal{E}))$ are densely defined closed symmetric bilinear forms on $\mathcal{H}_n$ and $\mathcal{H}$, respectively. We extend the definitions of $\mathcal{E}^n(u,u)$ to every $u\in\mathcal{H}_n$ by defining $\mathcal{E}^n(u,u)=\infty$ for $u\in\mathcal{H}_n\backslash\mathcal{D}(\mathcal{E}^n)$. Similar extension is done for $\mathcal{E}$ as well.

We assume throughout this section that for each $n\geq 1$, there is a bounded linear operator $E_n: \mathcal{H}_n\rightarrow\mathcal{H}$ such that $\pi_n$ is a left inverse of $E_n$, that is
\begin{equation}\label{eqep1}
\left\langle\pi_nf, f_n\right\rangle_n=\left\langle f, E_nf_n\right\rangle\mbox{ and }\pi_nE_nf_n=f_n\mbox{ for every }f\in\mathcal{H}, f_n\in\mathcal{H}_n.
\end{equation}

Moreover we assume that $\pi_n :\mathcal{H}\rightarrow\mathcal{H}_n$ satisfies the following two conditions
\begin{equation}\label{eqep2}
\sup_{n\geq 1}\|\pi_n\|<\infty\quad\text{and}\quad\lim_{n\rightarrow\infty}\|\pi_nf\|_n=\| f\|\quad\text{for every}\quad f\in\mathcal{H}.
\end{equation}

Let $\|E_n\|$ denote the operator norm. Note that $\left\langle E_nf_n, E_ng_n\right\rangle=\left\langle f_n,g_n\right\rangle_n$ for every $f_n,g_n\in\mathcal{H}_n, n\geq 1$ and so clearly
\begin{equation}\label{eqep3}
\|E_n\|=1\quad\text{and}\quad\|E_nf_n\|=\|f_n\|_n\quad\text{for every}\quad f_n\in\mathcal{H}_n, n\geq 1.
\end{equation}

{\it Definition.} Under the above setting, we say that $\mathcal{E}^n$ is Mosco-convergent to $\mathcal{E}$ in the generalized sense if
the following two conditions are satisfied.
\begin{itemize}
\item[(a)] If $v_n\in\mathcal{H}_n, u\in\mathcal{H}$ and $E_nv_n\rightarrow u$ weakly in $\mathcal{H}$, then
$$
\liminf_{n\rightarrow\infty}\mathcal{E}^n(v_n,v_n)\geq\mathcal{E}(u,u).
$$
\item[(b)] For every $u\in\mathcal{H}$, there exists $u_n\in\mathcal{H}_n$ such that $E_nu_n\rightarrow u$ strongly in $\mathcal{H}$ and
$$
\limsup_{n\rightarrow\infty}\mathcal{E}^n(u_n,u_n)\leq\mathcal{E}(u,u).
$$
\end{itemize}

In our case, since the state space of $X^n$ is $I_n$ while $X$ has $I$ as its state space, we need to define
the transforms between the functions on $I_n$ and $I$. First, if $v$ is defined on $I_n$
, let $E_nv$ be the extension of $v$ to $I$ defined by
\begin{equation*}
E_nv(x):=
v(a_i),\mbox{ $x\in [a_i,a_{i+1}), 1\leq i\leq n$}.
\end{equation*}
Besides, if $u\in L^2(I)$,
\begin{equation*}
\pi_nu(x):=
\dfrac{1}{m_n(\{a_i\})}\int_{a_i}^{a_{i+1}}u(z)m(dz), \mbox{ $x=a_i, 1\leq i\leq n$}
\end{equation*}

If $u\in C(I)$, define $R_nu$ to be the restriction of $u$ to $I_n$, i.e.
$$R_nu(x) :=u(x), \quad x\in I_n.$$

It is easy to check that  $E_n$ and $\pi_n$ defined above satisfy the condition of \eqref{eqep1} and \eqref{eqep2}. Let $L^2(I_n,m_n)$ and $L^2(I)$ correspond to $\mathcal{H}_n$ and $\mathcal{H}$ respectively. Notations such as inner product and norm keep the same. It is clear that for $u, v\in C(I)$, it holds that
\begin{equation}\label{eqrp}
\lim_{n\rightarrow\infty}\left\langle\pi_nu, \pi_nv\right\rangle_n=\lim_{n\rightarrow\infty}\left\langle\pi_nu, R_nv\right\rangle_n=\lim_{n\rightarrow\infty}\left\langle R_nu, R_nv\right\rangle_n.
\end{equation}

\begin{thm}\label{MOSCO}
Let $(\mathscr{E}^n,\mathscr{F}^n)$ and $(\mathscr{E}^{(\ss)},\mathscr{F}^{(\ss)})$ be the Dirichlet forms in \eqref{eqd} and \eqref{eqc1}. Then $\mathscr{E}^n$ converges to $\mathscr{E}^{(\ss)}$ in the generalized sense of Mosco.
\end{thm}

\begin{proof}
First, let us check Definition (a). Let $\varphi_n\in L^2(I_n), \varphi\in L^2(I)$, and $E_n\varphi_n$ weakly converge to $\varphi$ in $L^2(I)$. Suppose $\liminf\limits_{n\rightarrow\infty}\mathscr{E}^n(\varphi_n,\varphi_n)<\infty$ and $\varphi_n\in\mathscr{F}^n$ without loss of generality.

For each $\varphi_n$, there exists $u_n\in\mathscr{F}_e^{(s)}$, such that $\varphi_n=u_n \mbox{ $m_n$-a.e. on } I_n$ and $\mathscr{E}^n(\varphi_n,\varphi_n)=\mathscr{E}^{(s)}(H_{I_n}u_n,H_{I_n}u_n)$. We now prove that $H_{I_n}u_n$ also converges weakly to $\varphi$ in $L^2(I)$. Since $\liminf\limits_{n\rightarrow\infty}\mathscr{E}^{(s)}(H_{I_n}u_n,H_{I_n}u_n)<\infty$, there exists a subsequence (still denoted by) $\bigg\{\dfrac{dH_{I_n}u_n}{d\ss}\bigg\}$ bounded by a finite constant $M$ in $L^2(I)$.
For $x\in [a_i,a_{i+1}], 1\leq i\leq n$, we have $$H_{I_n}u_n(x)-E_n\varphi_n(x)=\int^x_{a_i}\dfrac{dH_{I_n}u_n}{d\ss}d\ss.$$ It follows that
\begin{align*}
|H_{I_n}u_n(x)-E_n\varphi_n(x)|&\leq\int_{a_i}^x\bigg|\dfrac{dH_{I_n}u_n}{d\ss}\bigg|d\ss\\
&\leq\big(\ss(x)-\ss(a_i)\big)^{1/2}\bigg(\int_{a_i}^x\bigg|
\dfrac{dH_{I_n}u_n}{d\ss}\bigg|^2d\ss\bigg)^{1/2}\\
&\leq M \big(\ss(x)-\ss(a_i)\big)^{1/2}.
\end{align*}
Since $\ss$ is uniformly continuous on $I$, for any $\varepsilon>0$, there exists $N$ large such that $|\ss(x)-\ss(y)|<\varepsilon$ whenever $|x-y|<\Delta^N$. Then for every $\phi\in L^2(I), n\geq N$, we get
\begin{align*}
&\left|\int_I\phi(x)\big(H_{I_n}u_n(x)-E_n\varphi_n(x)\big)m(dx)\right|\\
&\leq\sum_{i=1}^{n}\int_{a_i}^{a_{i+1}}\big|\phi(x)\big(H_{I_n}u_n(x)-E_n\varphi_n(x)\big)\big|m(dx)\\
&\leq M\sum_{i=1}^{n}\int_{a_i}^{a_{i+1}}|\phi(x)|\big(\ss(x)-\ss(a_i)\big)^{1/2}m(dx)\\
&\leq\varepsilon^{1/2}M\int_I |\phi(x)|m(dx) \leq\varepsilon^{1/2}M\sqrt{m(I)}\|\phi\|.
\end{align*}
Since $\varepsilon$ is arbitrary, we deduce that $H_{I_n}u_n\rightarrow \varphi$ weakly in $L^2(I)$.
Therefore $\liminf\limits_{n\rightarrow\infty}\|H_{I_n}u_n\|_{\mathscr{E}_1^{(s)}}<\infty$. Then there exists a subsequence, still denoted by $\{H_{I_n}u_n\}$ $\mathscr{E}_1^{(s)}$-weakly converging to a unique $v\in\mathscr{F}^{(s)}$.

For each $g\in L^2(I)$, we have
$$\mathscr{E}_1^{(s)}(G_1g,H_{I_n}u_n)=\langle g, H_{I_n}u_n\rangle.$$
By letting $n\rightarrow\infty$, it follows $\mathscr{E}_1^{(s)}(G_1g,v)=\langle g,v\rangle=\langle g,\varphi\rangle.$ Hence $v=\varphi$, $m$-a.e. Besides, for $f\in\mathscr{F}^{(s)}$,
$\mathscr{E}^{(s)}(f,H_{I_n}u_n)=\mathscr{E}_1^{(s)}(f,H_{I_n}u_n)-\langle f,H_{I_n}u_n\rangle$. From the fact that $H_{I_n}u_n\rightarrow\varphi$ weakly in $L^2(I)$ and $\mathscr{F}^{(s)}$, it follows $\dfrac{dH_{I_n}u_n}{d\ss}\rightarrow\dfrac{d\varphi}{d\ss}$ weakly in $L^2(I,d\ss)$.
Therefore,
\begin{align*}
\liminf\limits_{n\rightarrow\infty}\mathscr{E}^n(\varphi_n,\varphi_n)&=\liminf\limits_{n\rightarrow\infty}\mathscr{E}^{(s)}(H_{I_n}u_n,H_{I_n}u_n)=\liminf\limits_{n\rightarrow\infty}\frac{1}{2}\int_I\bigg(\dfrac{dH_{I_n}u_n}{d\ss}\bigg)^2d\ss\\
&\geq \frac{1}{2}\int_I\bigg(\dfrac{d\varphi}{d\ss}\bigg)^2d\ss=\mathscr{E}^{(s)}(\varphi,\varphi).
\end{align*}

Next, let us verify Definition (b). Suppose $u\in\mathscr{F}^{(\ss)}$ without loss of generality. Then $u\ll \ss$. Define $v_n :=R_nu$. It is obvious that $E_nv_n\rightarrow u$ strongly in $L^2(I)$. Besides,
$$
\limsup_{n\rightarrow\infty}\mathcal{E}^n(v_n,v_n)=\limsup_{n\rightarrow\infty}\mathcal{E}^{(s)}(H_{I_n}u,H_{I_n}u)\leq\mathscr{E}^{(\ss)}(u,u).
$$
The last inequality follows from \eqref{eqo}.
\end{proof}

Let $\{T_t^n, t\geq 0\}$ and $\{G_\lambda^n, \lambda>0\}$ be the strongly continuous symmetric contraction semigroup and the resolvent associated with $(\mathscr{E}^n,\mathcal{F}^n)$. Similarly, the semigroup and resolvent associated with $(\mathscr{E}^{(\ss)},\mathscr{F}^{\ss})$ will be denoted by $\{T_t, t\geq 1\}$ and $\{G_\lambda, \lambda>0\}$ respectively.
The following equivalence theorem is referred from \cite{ZPT}.

\begin{thm}\label{thmm}
Under the above setting, the followings are equivalent.
\begin{itemize}
\item[\rm (a)] $\mathscr{E}^n$ is Mosco-convergent to $\mathscr{E}$ in the generalized sense;
\item[\rm (b)] $E_nT_t^n\pi_n\rightarrow T_t$ strongly in $\mathcal{H}$ and the convergence is uniform in any finite interval of $t>0$;
\item[\rm (c)] $E_nG_\lambda^n\pi_n\rightarrow G_\lambda$ strongly in $\mathcal{H}$ for every $\lambda>0$.
\end{itemize}
\end{thm}

A main corollary of Mosco convergence is the following convergence of resolvent.
\def\EE{\mathscr{E}}

\begin{cor}\label{lemrn}
For $u\in C_c(I)$, it holds that
$$
\lim_n\|R_nG_\lambda u-G_\lambda^n\pi_nu\|_{\mathscr{E}_1^n}=0.
$$

\end{cor}
Note that $\|\cdot\|_{\mathscr{E}_1}$ denotes the $\mathscr{E}_1$-norm and $G_\lambda u\in\mathscr{F}^{(\ss)}$ and $\|G_\lambda u\|_\infty\leq \lambda^{-1}\|u\|_\infty$, so that $G_\lambda u\in C(I)$, $R_nG_\lambda u$ is well defined.

\begin{proof} It follows from the relation between $\EE^n$ and its resolvent $G^n_{\lambda}$  that
 \begin{align*}&\EE^n(R_nG_\lambda u-G_\lambda^n\pi_nu,R_nG_\lambda u-G_\lambda^n\pi_nu) \\&=\EE^n(R_nG_\lambda u, R_n G_{\lambda u})-2\langle R_nG_\lambda u,\pi_nu-\lambda G^n_{\lambda}\pi_nu\rangle_n
+\langle G_\lambda^n\pi_n u,\pi_nu-\lambda G_\lambda^n\pi_nu\rangle_n\end{align*}
It is obvious from Theorem \ref{MOSCO} that
$\mathscr{E}^n(R_nG_\lambda u ,R_nG_\lambda u)\rightarrow\mathscr{E}^{(\ss)}(G_\lambda u ,G_\lambda u)$.
By \eqref{eqrp},\eqref{eqep1}, \eqref{eqep2},
\eqref{eqep3} and Theorem \ref{thmm}, two inner products above have the same limit
$\langle G_{\lambda}u, u-\lambda G_{\lambda} u\rangle.$
Hence \begin{align*}\lim_n\mathscr{E}^n&(R_nG_\lambda u-G_\lambda^n\pi_nu,R_nG_\lambda u-G_\lambda^n\pi_nu)\\
&=\mathscr{E}^{(\ss)}(G_\lambda u, G_\lambda u)-\left\langle  G_\lambda u, u-\lambda G_{\lambda}u\right\rangle\\
&=\EE^{(\ss)}_{\lambda}(G_{\lambda}u,G_{\lambda}u)-\langle G_{\lambda}u,u\rangle=0.\end{align*}
Similarly, we can deduce that $\|R_nG_\lambda u-G_\lambda^n\pi_nu\|_n\rightarrow 0$.
\end{proof}

\section{Proof of Theorem \ref{thmwc}}\label{s4}
In this section, we complete the proof of weak convergence by providing some tightness results. The idea of it is due to \cite{P}.
\begin{proof}[Proof of Theorem \ref{thmwc}]
We complete the proof of the main theorem according to the following steps.

Step 1. First, we show that for every $\lambda>0, T>0$ and $u\in C_c(I)$, it holds that
\begin{equation}\label{eqrn}
\limsup_{n\rightarrow\infty}\mathbb{E}_{m_n}\bigg[\sup_{t\in [0,T]}\big|G_\lambda^n\pi_nu(X_t^n)-R_nG_\lambda u(X_t^n)\big|\bigg]=0.
\end{equation}
Fix $\lambda,T>0$.
Given $\varepsilon>0$, let
$$
D^n=\{x\in I_n; \big|G_\lambda^n\pi_nu(x)-R_nG_\lambda u(x)\big|>\varepsilon\}, \sigma_{D^n} =\inf\{t>0; X^n_t\in D^n\}.
$$
The left side in \eqref{eqrn} is less than the sum of
$$
M_1 :=\limsup\limits_{n\rightarrow\infty}\mathbb{E}_{m_n}\bigg[\sup_{t\in [0,T]}\big|G_\lambda^n\pi_nu(X_t^n)-R_nG_\lambda u(X_t^n)\big|;\sigma_{D^n}>T\bigg]
$$
and
$$
M_2 :=\limsup\limits_{n\rightarrow\infty}\mathbb{E}_{m_n}\bigg[\sup_{t\in [0,T]}\big|G_\lambda^n\pi_nu(X_t^n)-R_nG_\lambda u(X_t^n)\big|;\sigma_{D^n}\leq T\bigg].
$$
It is obvious that $M_1\leq\varepsilon$. As for $M_2$, if we set $p^1_{D^n}(\cdot):=\mathbb{E}_{\cdot}^n\big[e^{-\sigma_{D^n}}\big]$, we have
\begin{align}
\mathbb{P}_{m_n}\big[\sigma_{D^n}&\leq T\big]\leq e^{T}\langle 1,p^1_{D^n}\rangle_n\leq e^T\sqrt{m(I)}\cdot\|p^1_{D^n}\|_n\notag\\
&\leq e^Tc\text{Cap}^n(D^n)^{1/2}\leq e^Tc\varepsilon^{-1}\|G_\lambda^n\pi_nu(x)-R_nG_\lambda u(x)\|_{\EE^n_1}. \label{lemma4.2}
\end{align}
by the definition of the capacity (see \cite[\S 2.1]{MYM}), where $c=\sqrt{m(I)}$. Therefore $$M_2\leq\frac{2\|u\|_\infty}{\lambda}\limsup\limits_{n\rightarrow\infty}\mathbb{P}_{m_n}\big[\sigma_{D^n}\leq T\big]=0$$ by Corollary \ref{lemrn}. Since $\varepsilon$ is arbitrary, \eqref{eqrn} follows.

Step 2. Let $f\in C_c(I)$. We next prove that for any $T>0, \varepsilon>0$, there exist $\lambda_0>0$ and $u\in C_c(I)\cap\mathscr{F}^{(\ss)}$, such that
\begin{equation}\label{equ}
\limsup_{n\rightarrow\infty}\mathbb{E}_{m_n}\bigg[\sup_{t\in [0,T]}\big| f(X_t^n)-\lambda_0G_{\lambda_0}^n\pi_nu(X^n_t)\big|\bigg]<\varepsilon.
\end{equation}
Fix $f\in C_c(I),T,\varepsilon >0.$  Since $\mathscr{E}^{(\ss)}$ is regular, there exists $u\in C_c(I)\cap\mathscr{F}^{(s)}$, such that
\begin{equation}\label{eqf}
\sup_{x\in I}|f(x)-u(x)|<\frac{\varepsilon}{4}.
\end{equation}
Denote
$$
F^n=\{x\in I_n; \big|R_n(u-\lambda G_\lambda u)(x)\big|>\dfrac{\varepsilon}{2}\}, \sigma_{F^n} =\inf\{t>0; X^n_t\in F^n\}.
$$
Let
$$
N_1 :=\limsup\limits_{n\rightarrow\infty}\mathbb{E}_{m_n}\bigg[\sup_{t\in [0,T]}\big|R_n(u-\lambda G_\lambda u)(X_t^n)\big|;\sigma_{F^n}>T\bigg]
$$
and
$$
N_2 :=\limsup\limits_{n\rightarrow\infty}\mathbb{E}_{m_n}\bigg[\sup_{t\in [0,T]}\big|R_n(u-\lambda G_\lambda u)(X_t^n)\big|;\sigma_{F^n}\leq T\bigg].
$$
It is obvious that $N_1\leq\dfrac{\varepsilon}{2}$. As for $N_2$, in a similar way of \eqref{lemma4.2}, we have 
\begin{align}
\mathbb{P}_{m_n}\big[\sigma_{F^n}\leq T]&\leq 2e^Tc\varepsilon^{-1}\|R_n(u-\lambda G_\lambda u)(x)\|_{\EE^n_1}. 
\end{align}
Therefore,
\begin{align}
N_2&\leq 2\|u\|_\infty\limsup\limits_{n\rightarrow\infty}\mathbb{P}_{m_n}\big[\sigma_{F^n}\leq T\big]\\
&\leq 4e^Tc\varepsilon^{-1}\|u\|_\infty\limsup_{n\rightarrow\infty}\|R_n(u-\lambda G_\lambda u)(x)\|_{\EE^n_1}\\
&\leq 4e^Tc\varepsilon^{-1}\|u\|_\infty \|(u-\lambda G_\lambda u)(x)\|_{\EE^{(\ss)}_1},
\end{align}
where the last inequality follows from the proof of Mosco convergence. For $u\in\mathscr{F}^{(\ss)}$, we know that $\lambda G_\lambda u\rightarrow u$ in $\mathscr{E}^{(\ss)}_1$-norm by \cite[Lemma 1.3.3]{MYM}. Then choose $\lambda_0$ such that $N_2<\dfrac{3\varepsilon}{4}.$ It follows that
\begin{equation}\label{eqf1}
\limsup_{n\rightarrow\infty}\mathbb{E}_{m_n}\bigg[\sup_{t\in [0,T]}\big|R_n(u-\lambda_0 G_{\lambda_0} u)(X_t^n)\big|\bigg]<\frac{3}{4}\varepsilon.
\end{equation}
Therefore, by \eqref{eqrn}\eqref{eqf}\eqref{eqf1}, \eqref{equ} holds.

Step 3. In this step, we will demonstrate that for any finite $m\geq 1$ and $\{f_1,...,f_m\}\subset C_c(I), \{(f_1,...,f_m)(X^n)\}_{n\geq 1}$ under the condition that  $X^n$ takes $m_n$ as initial distribution forms a tight family on $\mathbb{D}_{\mathbb{R}^m}[0,\infty)$.

It suffices to consider $m=1$ and $f :=f_1$ for the sake of brevity. Fix $\varepsilon, T>0$, apply Step 2 to these $f, \varepsilon, T$ and choose $u,\lambda_0$ accordingly. Set $Y_t^n  :=\lambda_0G^n_{\lambda_0}\pi_nu(X_t^n)$ such that \eqref{equ} holds. Set $Z_t^n :=\lambda_0(\lambda_0G^n_{\lambda_0}\pi_nu-\pi_nu)(X_t^n)$.

From Fukushima's decomposition of $X^n$ with respect to $\lambda_0G^n_{\lambda_0}\pi_nu$ (see \cite[Theorem 5.2.2]{MYM}), one can find that
$$
t\mapsto Y_t^n-\int_0^tZ_s^nds
$$
is a martingale relative to the filtration of $X^n$.  Besides, \eqref{equ} yields that
$$
\limsup_{n\rightarrow\infty}\mathbb{E}_{m_n}\bigg[\sup_{t\in [0,T]}\big| f(X_t^n)-Y_t^n\big|\bigg]<\varepsilon.
$$
Furthermore, it follows that
$$
\limsup_{n\rightarrow\infty}\mathbb{E}_{m_n}\bigg[\sup_{t\in [0,T]}|Z_t^n|\bigg]\leq 2\lambda_0\|u\|_\infty<\infty.
$$
Therefore, \cite[Theorem 3.9.4,Remark 3.9.5(b)]{ST} yields the conclusion.

Step 4. Since $C_c(I)$ strongly separates points in $I$(for the definition of strong separation, see \cite[\S3.4,(4.7)]{ST}; for the proof of this fact, see \cite{DM}), by\cite[Corollary 3.9.2]{ST}, it only remains to show that for any finite $m\geq 1$ and $\{f_1,...,f_m\}\subset C_c(I), (f_1,...,f_m)(X^n)$ weakly converge to $(f_1,...,f_m)(X)$ with $X_n$ and $X$ having $m_n$ and $m$ as their initial distribution respectively. To this end, take $g_1\in C_b(\mathbb{R}^m)$ and set $h_1 :=g_1\circ (f_1,...,f_m)\in C_b(I).$ Given any $t_1,...,t_p>0$, by the contraction of semigroup and uniform continuity of $h_1(x)$, we have
$$
\lim_{n\rightarrow\infty}\mathbb{E}_{m_n}\big[h_1(X_{t_1}^n)\big]=\lim_{n\rightarrow\infty}\int_IE_nT_{t_1}^nR_nh_1(x)m(dx)=\lim_{n\rightarrow\infty}\int_IE_nT_{t_1}^n\pi_nh_1(x)m(dx),$$
which converge to $\mathbb{E}_{m}\big[h_1(X_{t_1})\big]$ by Theorem \ref{thmm}. In fact, the last equality is deduced from the following reasons. Since $h_1(x)$ is uniformly continuous on $I$, for any $\varepsilon>0$, there exists $N$ large such that for any $n>N$, $|R_nh_1(x)-\pi_nh_1(x)|<\varepsilon, \forall x\in I_n.$ Therefore, for $n\geq N$, by the contraction of semigroup,
\[
\bigg|\int_IE_nT_{t_1}^n\big(R_nh_1(x)-\pi_nh_1(x)\big)m(dx)\bigg|\leq \varepsilon m(I).
\]
Since $\varepsilon$ is arbitrary, the last equality holds.

Inductively (see \cite{P} for more details), we conclude that
$$
\lim_{n\rightarrow\infty}\mathbb{E}_{m_n}\big[h_1(X_{t_1}^n)\cdot\cdot\cdot h_p(X_{t_p}^n)\big]=\mathbb{E}_{m}\big[h_1(X_{t_1})\cdot\cdot\cdot h_p(X_{t_p})\big].
$$
Combined with the tightness of $\{(f_1,...,f_m)(X^n)\}_{n\geq 1}$ we deduce the result.
\end{proof}

{\it Remark.}
Our proof of tightness results is relatively general. So we conclude a useful result. Given a bounded domain $D\subset\mathbb{R}^d$ and a sequence of subsets $D_n\subset D,n\geq 1$. Radon measure $m$ and $m_n$ are on $D$ and $D_n$ respectively with $m_n$ weakly converging to $m$ on $D$. For $n\geq 1$, $X^n$ is the stochastic process on $D_n$ corresponding to the regular Dirichlet form $(\mathscr{E}^n,\mathscr{F}^n)$ on $L^2(D_n,m_n)$. $X$ is the stochastic process on $D$ corresponding to the regular Dirichlet form $(\mathscr{E},\mathscr{F})$ on $L^2(D,m)$. Use the setting of the generalized version of Mosco convergence in \S\ref{s4}. $\mathcal{H} :=L^2(D,m)$ and $\mathcal{H}_n :=L^2(D_n,m_n)$. If the following two conditions are satisfied:
\begin{itemize}
\item[(a)] If $v_n\in\mathcal{H}_n, u\in\mathcal{H}$ and $E_nv_n\rightarrow u$ weakly in $\mathcal{H}$, then
$$
\liminf_{n\rightarrow\infty}\mathscr{E}^n(v_n,v_n)\geq\mathscr{E}(u,u).
$$
\item[(b)] For every $u\in\mathcal{H}$,
$\limsup\limits_{n\rightarrow\infty}\mathscr{E}^n(\pi_nu,\pi_nu)\leq\mathscr{E}(u,u).$
\end{itemize}
Then $(X^n, \mathbb{P}^n_{m_n})$ converges weakly to $(X, \mathbb{P}_m)$ on $\mathbb{D}_{D}[0,\infty)$ equipped with the Skorohod topology.

\noindent
{\it Acknowledgement.} The authors would like to thank Yushu, Zheng, PhD student of the second
author, for his helpful suggestions.

\end{document}